\documentclass[11pt]{article}
\usepackage{latexsym,amssymb,amsmath,amsfonts,mathrsfs,amsthm}
\usepackage{enumerate,graphics,graphicx,color,dsfont}
\usepackage{subfigure,txfonts,marginnote,bm,multirow}
\newcommand{\R}{{\mathbb R}}
\newcommand{\Z}{{\mathbb Z}}

\newcommand{\Sp}{{\mathbb S}}
\newcommand{\ds}{\displaystyle}
\newcommand{\no}{\nonumber}
\newcommand{\be}{\begin{eqnarray}}
\newcommand{\ben}{\begin{eqnarray*}}
\newcommand{\en}{\end{eqnarray}}
\newcommand{\enn}{\end{eqnarray*}}

\newcommand{\pa}{\partial}

\newcommand{\ov}{\overline}
\newcommand{\curl}{{\rm curl\,}}

\newcommand{\I}{{\rm Im}}
\newcommand{\Rt}{{\rm Re}}
\newcommand{\g}{\gamma}
\newcommand{\G}{\Gamma}

\newcommand{\vep}{\varepsilon}

\newcommand{\om}{\omega}

\newcommand{\wid}{\widetilde}

\newcommand{\ra}{\rightarrow}
\newcommand{\se}{\setminus}

\definecolor{xxx}{rgb}{0,0,0}
\newtheorem{theorem}{Theorem}[section]
\newtheorem{lemma}[theorem]{Lemma}

\begin{document}
\renewcommand{\theequation}{\arabic{section}.\arabic{equation}}

\title{\bf Uniqueness in inverse electromagnetic scattering problem with phaseless far-field data
at a fixed frequency}
\author{Xiaoxu Xu\thanks{Academy of Mathematics and Systems Science, Chinese Academy of Sciences,
Beijing 100190, China and School of Mathematical Sciences, University of Chinese
Academy of Sciences, Beijing 100049, China ({\tt xuxiaoxu14@mails.ucas.ac.cn})}
\and
Bo Zhang\thanks{LSEC, NCMIS and Academy of Mathematics and Systems Science, Chinese Academy of
Sciences, Beijing, 100190, China and School of Mathematical Sciences, University of Chinese
Academy of Sciences, Beijing 100049, China ({\tt b.zhang@amt.ac.cn})}
\and
Haiwen Zhang\thanks{NCMIS and Academy of Mathematics and Systems Science, Chinese Academy of Sciences,
Beijing 100190, China ({\tt zhanghaiwen@amss.ac.cn})}
}
\date{}
\maketitle

\begin{abstract}
This paper is concerned with uniqueness in inverse electromagnetic scattering with phaseless far-field
pattern at a fixed frequency. In our previous work [{\em SIAM J. Appl. Math.} {\bf 78} (2018), 3024-3039],
by adding a known reference ball into the acoustic scattering system, it was proved that the impenetrable
obstacle and the index of refraction of an inhomogeneous medium can be uniquely determined by the
acoustic phaseless far-field patterns generated by infinitely many sets of superpositions of two plane
waves with different directions at a fixed frequency.
In this paper, we extend these uniqueness results to the inverse electromagnetic scattering case.
The phaseless far-field data are the modulus of the tangential component in the orientations $\bm e_\phi$
and $\bm e_\theta$, respectively, of the electric far-field pattern measured on the unit sphere and
generated by infinitely many sets of superpositions of two electromagnetic plane waves with different
directions and polarizations. Our proof is mainly based on Rellich's lemma and the Stratton--Chu formula
for radiating solutions to the Maxwell equations.

\vspace{.2in}
{\bf Keywords:} Uniqueness, inverse electromagnetic scattering, phaseless far-field pattern,
perfectly conducting obstacle, impedance obstacle, partially coated obstacle, inhomogeneous medium.
\end{abstract}

\section{Introduction}
\setcounter{equation}{0}

Inverse scattering theory has wide applications in such fields as radar, sonar, geophysics, medical
imaging, and nondestructive testing (see, e.g., \cite{CK,KG}). This paper is concerned with
inverse electromagnetic scattering by a bounded obstacles or an inhomogeneous medium from
phaseless far-field data, associated with incident plane waves at a fixed frequency.

Inverse scattering problems with phased data have been extensively studied both mathematically
and numerically in the past several decades (see, e.g., \cite{C18,CK,KG}). However, in many
applications, it is difficult to measure the phase of the wave field accurately, compared with
the modulus of the wave field. Therefore, it is desirable to reconstruct the scatterers
from the phaseless near-field or far-field data (i.e., the intensity of the near-field or far-field),
which is called the \emph{phaseless inverse scattering problem}.

The main difficulty of inverse scattering problems with phaseless far-field data is the so-called
\emph{translation invariance} property of the phaseless far-field pattern, that is, the modulus of
the far-field pattern generated by one plane wave is invariant under the translation of the scatterers.
This implies that it is impossible to recover the location of the scatterer from the phaseless
far-field data  with one plane wave as the incident field. Several iterative methods have been proposed
in \cite{I2007,IK2010,IK2011,KR} to reconstruct the shape of the scatterer. Under a priori condition that
the sound-soft scatterer is a ball or disk, it was proved in \cite{LZ10} that the radius of the scatterer
can be uniquely determined by a single phaseless far-field datum. It was proved in \cite{Majda} that
the shape of a general, sound-soft, strictly convex obstacle can be uniquely determined by the phaseless
far-field data generated by one plane wave at a high frequency.
However, there is no translation invariance property for phaseless near-field data. Therefore,
many numerical algorithms for inverse scattering problems with phaseless near-field data
have been developed (see, e.g., \cite{C18,CFH17,CH17,KR16,S16,XZZ3} for the acoustic case
and \cite{CH16} for the electromagnetic case). Uniqueness results and stability have also been established
for inverse scattering problems with phaseless near-field data
(see \cite{K14,K17,KR17,MH17,N15,N16,RY18,XZZ19,ZSGL19,ZWGL19} for the acoustic and potential scattering
case and \cite{Romanov17,XZZ19} for the electromagnetic scattering case).

Recently in \cite{ZZ01}, it was proved that the translation invariance property of the phaseless
far-field pattern can be broken by using superpositions of two plane waves as the incident
fields with an interval of frequencies. Following this idea, several algorithms have been developed
for inverse acoustic scattering problems with phaseless far-field data, based on using the superposition
of two plane waves as the incident field (see \cite{ZZ01,ZZ02,ZZ03}). Further, by using the spectral
properties of the far-field operator, rigorous uniqueness results have also been established in \cite{XZZ1}
for inverse acoustic scattering problems with phaseless far-field data generated by infinitely many sets
of superpositions of two plane waves with different directions at a fixed frequency,
under certain a priori assumptions on the property of the scatterers.
In \cite{XZZ2}, by adding a known reference ball into the acoustic scattering system, it was shown that
the uniqueness results obtained in \cite{XZZ1} remain true without the a priori assumptions
on the property of the scatterers.
Note that the idea of adding a known reference ball to the scattering system was first applied
in \cite{LLZ09} to numerically enhance the reconstruction results of the linear sampling method
and also used in \cite{ZG18} to prove uniqueness in inverse acoustic scattering problems with phaseless
far-field data. Note further that there are certain studies on uniqueness for phaseless inverse
scattering problems with using superpositions of two point sources as the incident fields
(see \cite{RY18,SZG18,XZZ19,ZSGL19,ZWGL19}).

The purpose of this paper is to establish uniqueness results in inverse electromagnetic scattering
problems with phaseless far-field data at a fixed frequency, extending the uniqueness results
in \cite{XZZ2} for the acoustic case to the electromagnetic case.
Different from the acoustic case considered in \cite{XZZ2}, the electric far-field pattern is
a complex-valued vector function, so the measurement of the phaseless electric far-field pattern
is more complicated. In practice, one usually makes measurement of the modulus of each tangential
component of the electric total-field or electric far-field pattern on the measurement surface
(see, e.g., \cite{Hansen,Pan11,Schmidt,ZW19}).
Motivated by this and the idea in \cite{XZZ1,XZZ2}, we make use of superpositions of two
electromagnetic plane waves with different directions and polarizations as the incident fields
and consider the modulus of the tangential component in the orientations $\bm e_\phi$
and $\bm e_\theta$, respectively, of the corresponding electric far-field pattern measured on the
unit sphere as the measurement data (called the phaseless electric far-field data).
We then prove that, by adding a known reference ball into the electromagnetic scattering system,
the impenetrable obstacle or the refractive index of the inhomogeneous medium
(under the condition that the magnetic permeability is a positive constant)
can be uniquely determined by the phaseless electric far-field data at a fixed frequency.
Our proof is mainly based on Rellich's lemma and the Stratton--Chu formula for radiating solutions
to the Maxwell equations.

The rest part of this paper is organised as follows. In Section \ref{s2}, we introduce
the electromagnetic scattering problems considered. The uniqueness results for inverse obstacle
and medium electromagnetic scattering with phaseless electric far-field data are presented in
Section \ref{s-o} and \ref{s-m}, respectively. Conclusions are given in Section \ref{s5}.

\section{The electromagnetic scattering problems}\label{s2}

In this section, we introduce the electromagnetic scattering problems considered in this paper.
To give a precise description of the scattering problems, we assume that $D$ is an open and bounded
domain in $\R^3$ with $C^2-$boundary $\pa D$ satisfying that the exterior $\R^3\se\ov{D}$ of $D$ is
connected. Note that $D$ may not be connected and thus may consist of several (finitely many)
connected components. We consider the time-harmonic ($e^{-i\om t}$ time dependence) incident
electromagnetic plane waves described by the matrices $E^i(x,d)$ and $H^i(x,d)$ defined by
\be\label{inc}
&& E^i(x,d)p:=\frac ik{\rm curl}\,{\rm curl}\,pe^{ikx\cdot d}=ik(d\times p)\times de^{ikx\cdot d},
\quad x\in\R^3,\\ \label{inc+}
&& H^i(x,d)p:={\rm curl}\,pe^{ikx\cdot d}=ikd\times pe^{ikx\cdot d},\quad x\in\R^3,
\en
where $d\in\Sp^2$ is the incident direction with $\Sp^2$ being the unit sphere, $p\in\R^3$ is
the polarization vector, $k=\om/\sqrt{\varepsilon_0\mu_0}$ is the wave number,
$\om$ is the frequency, and $\varepsilon_0$ and $\mu_0$ are the electric permittivity and magnetic
permeability of a homogeneous medium, respectively.

When $D$ is an impenetrable obstacle, then the scattering problem can be modeled by the
exterior boundary value problem:
\begin{subequations}
\be\label{me1}
{\rm curl}\,E-ikH=0 &&\;\; \text{in}\;\;\R^3\se\ov{D},\\ \label{me2}
{\rm curl}\,H+ikE=0 &&\;\; \text{in}\;\;\R^3\se\ov{D},\\ \label{bc}
\mathscr BE=0 &&\;\; \text{on}\;\;\pa D,\\ \label{smrc}
\lim_{r\ra\infty}(H^s\times x-rE^s)=0, &&\;\; r=|x|,
\en
\end{subequations}
where $(E^s,H^s)$ is the scattered field, $E:=E^i+E^s$ and $H:=H^i+H^s$ are the electric
total-field and the magnetic total-field, respectively, the equations (\ref{me1})--(\ref{me2})
are the Maxwell equations, and (\ref{smrc}) is the Silver--M\"uller radiation condition.
The boundary condition $\mathscr B$ in (\ref{bc}) depends on the physical property of
the obstacle $D$, that is, $\mathscr{B}E=\nu\times E$ if $D$ is a perfectly conducting obstacle,
$\mathscr{B}E=\nu\times\curl{E}-i\lambda(\nu\times E)\times\nu=0$ if $D$ is an impedance obstacle,
and $\mathscr{B}E=\nu\times E$ on $\G_D$,
$\mathscr{B}E=\nu\times\curl{E}-i\lambda(\nu\times E)\times\nu=0$ on $\G_I$ if $D$ is a partially
coated obstacle,
where $\nu$ is the unit outward normal vector on the boundary $\pa D$. Here, for the case
when $D$ is an impedance obstacle, we assume that $\lambda$ is the impedance function on
$\pa D$ with $\lambda\in C(\pa D)$ and $\lambda(x)\ge 0$ for all $x\in\pa D$.
Further, for the case when $D$ is a partially coated obstacle, we assume that $\pa D$ has a
Lipschitz dissection $\pa D=\G_D\cup\Pi\cup\G_I$ with $\G_D$ and $\G_I$ being disjoint and
relatively open subsets of $\pa D$ and having $\Pi$ as their common boundary in $\pa D$
(see, e.g. \cite{M}) and $\lambda$ is the impedance function on $\G_I$ with $\lambda\in C(\G_I)$
and $\lambda(x)\ge 0$ for all $x\in\G_I$.

When $D$ is an inhomogeneous medium, we assume that the magnetic permeability $\mu=\mu_0$ is a positive
constant in the whole space. Then the scattering problem is modeled by the medium scattering problem
\begin{subequations}
\be\label{men1}
{\rm curl}\,E-ikH=0 &&\;\;\text{in}\;\;\R^3,\\ \label{men2}
{\rm curl}\,H+iknE=0 &&\;\;\text{in}\;\;\R^3,\\ \label{smrcn}
\lim_{r\ra\infty}(H^s\times x-rE^s)=0, &&\;\; r=|x|,
\en
\end{subequations}
where $(E^s,H^s)$ is the scattered field, $E:=E^i+E^s$ and $H:=H^i+H^s$ are the electric total-field
and the magnetic total-field, respectively. The refractive index $n$ in (\ref{men2}) is given by
\ben
n(x):=\frac1{\vep_0}\left(\vep(x)+i\frac{\sigma(x)}\om\right),
\enn
where $\vep(x)$ is the electric permittivity with $\vep(x)\geq\vep_{min}$ in $\R^3$
for a constant $\vep_{min}>0$ and
$\sigma(x)$ is the electric conductivity with $\sigma(x)\geq0$ in $\R^3$.
We assume further that $n-1$ has a compact support $\ov{D}$ and $n\in C^{2,\gamma}(\R^3)$ for $0<\gamma<1$.
From the above assumptions, it can be seen that
${\rm Re}[n(x)]\geq n_{min}:=\varepsilon_{min}/\varepsilon_0>0$ and ${\rm Im}[n(x)]\geq0$ for all $x\in\R^3$.

The existence of a unique (variational) solution to the problems (\ref{me1})--(\ref{smrc})
and (\ref{men1})--(\ref{smrcn}) has been proved in \cite{CCM,CCM2,CK81,CK,M}
(see Theorem 6.21 in \cite{CK} and Theorem 10.8 in \cite{M} for scattering by a perfectly
conducting obstacle or an impedance obstacle with $\lambda\equiv0$ on $\pa D$,
Theorems 6.11 and 9.11 in \cite{CK}
for scattering by an impedance obstacle with constant impedance function,
Theorems 2.1 and 3.3 in \cite{CK81} for scattering by an impedance obstacle with
$\lambda\in C^{0,\gamma}(\pa D)$, Theorem 3.5 in \cite{CCM2} (see also Theorem 2.7
in \cite{CCM}) for scattering by a partly coated obstacle or an impedance obstacle
and Theorem 5.5 in \cite{KG} (see also Theorem 9.5 in \cite{CK}) for scattering by an inhomogeneous medium).
In particular, it is well-known from \cite{CK} that the electric and magnetic scattered fields $E^s$
and $H^s$ have the asymptotic behavior
\ben
&&E^s(x,d)p=\frac{e^{ik|x|}}{|x|}\left\{E^\infty(\hat x,d)p+O\left(\frac1{|x|}\right)\right\},\;\;|x|\ra\infty,\\
&&H^s(x,d)p=\frac{e^{ik|x|}}{|x|}\left\{H^\infty(\hat x,d)p+O\left(\frac1{|x|}\right)\right\},\;\;|x|\ra\infty
\enn
uniformly for all observation directions $\hat x=x/|x|\in\Sp^2$, where
$E^\infty(\hat x,d)p$ is the electric far-field pattern of $E^s(x,d)p$ and $H^\infty(\hat x,d)p$
is the magnetic far-field pattern of $H^s(x,d)p$ for any $p\in\R^3$,
satisfying that (see \cite[(6.24)]{CK})
\be\label{EfHf}
H^\infty(\hat x,d)p=\hat x\times E^\infty(\hat x,d)p,\;\;\;
\hat x\cdot E^\infty(\hat x,d)p=\hat x\cdot H^\infty(\hat x,d)p=0.
\en
Because of the linearity of the direct scattering problem with respect to the incident field,
the scattered waves, the total-fields and the corresponding far-field patterns can be represented
by matrices $E^s(x,d)$ and $H^s(x,d)$, $E(x,d)$ and $H(x,d)$, and $E^\infty(\hat x,d)$ and
$H^\infty(\hat x,d)$, respectively. Each component of the matrices $E^\infty(\hat x,d)$ and
$H^\infty(\hat x,d)$ is an analytic function of $\hat x\in\Sp^2$ for each $d\in\Sp^2$ and
of $d\in\Sp^2$ for each $\hat x\in\Sp^2$ (see, e.g., \cite{CK}).

Throughout this paper, we assume that the wave number $k$ is arbitrarily fixed, i.e.,
the frequency $\om$ is arbitrarily fixed. Following \cite{XZZ1,XZZ2,ZZ01,ZZ02}, we make use
of the following superposition of two plane waves as the incident (electric) field:
\ben
E^i:=E^i(x,d_1,d_2,p_1,p_2)=E^i(x,d_1)p_1+E^i(x,d_2)p_2
=\frac ik{\rm curl}\,{\rm curl}\,p_1e^{ikx\cdot d_1}
+\frac ik{\rm curl}\,{\rm curl}\,p_2e^{ikx\cdot d_2},
\enn
where $d_1,d_2\in\Sp^2$ and $p_1,p_2\in\R^3$. Then the (electric) scattered field $E^s$ has the
asymptotic behavior
\ben
E^s(x,d_1,d_2,p_1,p_2)=\frac{e^{ik|x|}}{|x|}\left\{E^\infty(\hat x,d_1,d_2,p_1,p_2)+O\left(\frac1{|x|}\right)\right\},\;|x|\ra\infty
\enn
uniformly for all observation directions $\hat x\in\Sp^2$. From the linear superposition principle
it follows that
\ben
E^s(x,d_1,d_2,p_1,p_2)=E^s(x,d_1)p_1+E^s(x,d_2)p_2
\enn
and
\be\label{lsp}
E^\infty(\hat x,d_1,d_2,p_1,p_2)=E^\infty(\hat x,d_1)p_1+E^\infty(\hat x,d_2)p_2,
\en
where $E^s(x,d_j)p_j$ and $E^\infty(\hat x,d_j)p_j$ are the (electric) scattered field and its
far-field pattern corresponding to the incident electric field $E^i(x,d_j)p_j$, respectively, $j=1,2$.

Following the idea in \cite{Hansen,Pan11,Schmidt,ZW19}, we measure the modulus of the tangential
component of the electric far-field pattern on the unit sphere $\Sp^2$. To present the tangential
components, we introduce the spherical coordinates
\ben
\begin{cases}
\hat x_1=\sin\theta\cos\phi,\\
\hat x_2=\sin\theta\sin\phi,\\
\hat x_3=\cos\theta,
\end{cases}
\enn
with $\hat x:=(\hat x_1,\hat x_2,\hat x_3)\in\Sp^2$ and $(\theta,\phi)\in[0,\pi]\times[0,2\pi)$.
For any $\hat x\in\Sp^2\se\{N,S\}$, the spherical coordinates give an one-to-one correspondence
between $\hat x$ and $(\phi,\theta)$, where $N:=(0,0,1)$ and $S:=(0,0,-1)$ denote the north and
south poles of $\Sp^2$, respectively. Define
\ben
\bm e_\phi(\hat x):=(-\sin\phi,\cos\phi,0),\;\;\;\;
\bm e_\theta(\hat x):=(\cos\theta\cos\phi,\cos\theta\sin\phi,-\sin\theta).
\enn
Then $\bm e_\phi(\hat{x})$ and $\bm e_\theta(\hat{x})$ are two orthonormal tangential vectors of
$\Sp^2$ at $\hat x\notin\{N,S\}$. Thus the phaseless far-field data we use are
$|\bm e_m(\hat{x})\cdot E^\infty(\hat{x},d_1,d_2,p_1,p_2)|$, $\hat x\in\Sp^2\se\{N,S\}$,
$m\in\{\phi,\theta\}$, $d_j\in\Sp^2$ and $p_j\in\R^3$ such that $d_j\bot p_j$, $j=1,2$.

{\em The inverse electromagnetic obstacle (or medium) scattering problem} we consider in this paper
is to reconstruct the obstacle $D$ and its physical property (or the refractive
index $n$ of the inhomogeneous medium) from the phaseless far-field data
$|\bm e_m(\hat{x})\cdot E^\infty(\hat{x},d_1,d_2,p_1,p_2)|$, $\hat x\in\Sp^2\se\{N,S\}$,
$m\in\{\phi,\theta\}$, $d_j\in\Sp^2$ and $p_j\in\R^3$ such that $d_j\bot p_j$, $j=1,2$.
The purpose of this paper is to establish the uniqueness results for these inverse problems.

\section{Uniqueness for inverse electromagnetic obstacle scattering}\label{s-o}
\setcounter{equation}{0}

Let $B$ be a given, perfectly conducting ball and let us assume that $k$ is not a Maxwell eigenvalue
in $B$. Here, $k$ is called a Maxwell eigenvalue in $B$ if the electromagnetic interior boundary
value problem
\be\label{eq6}
{\rm curl}\,\wid E-ik\wid H=0 && \text{in}\;\;B,\\ \label{eq6+}
{\rm curl}\,\wid H+ik\wid E=0 && \text{in}\;\;B,\\ \label{eq7}
\nu\times\wid E=0 && \text{on}\;\;\pa B
\en
has a nontrivial solution $(\wid E,\wid H)$. Note that, if the radius $r$ of the ball $B$ is chosen
so that $j_n(kr)\neq0$ and $j_n(kr)+krj'_n(kr)\neq0$ for $n=0,1,\ldots$, then
$k$ is not a Maxwell eigenvalue in $B$ (see \cite[Page 252]{CK}), where $j_n$ denotes the spherical
Bessel function of order $n$.

Now, denote by $E_j^s$, $H_j^s$, $E_j^\infty$, and $H_j^\infty$ the electric scattered field,
the magnetic scattered field, the electric far-field pattern and the magnetic far-field pattern,
respectively, associated with the obstacle $D_j\cup B$ and corresponding to the incident
electromagnetic waves $E^i$ and $H^i$, $j=1,2$. The geometry of the scattering problem is given in
Figure \ref{figo}. Then we have the following uniqueness result for the inverse electromagnetic obstacle
problem.
\begin{figure}[!ht]
\centering
\includegraphics[width=0.5\textwidth]{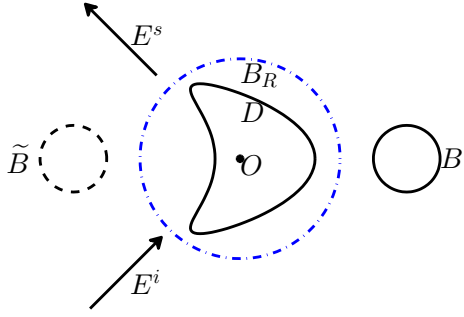}
\vspace{-1cm}
\caption{Scattering by a bounded obstacle.
}\label{figo}
\end{figure}

\begin{theorem}\label{o}
Assume that $B$ is a given perfectly conducting reference ball such that $k$ is not a Maxwell
eigenvalue in $B$. Suppose $D_1$ and $D_2$ are two obstacles with $\ov{D_1\cup D_2}\subset B_R$,
where $B_R$ is a ball of radius $R$ and centered at the origin satisfying that
$\ov{B}\cap\ov{B_R}=\emptyset$. If the corresponding electric far-field patterns satisfy that
\be\label{m=}
|\bm e_m(\hat x)\cdot E_1^\infty(\hat x,d_1,d_2,p_1,p_2)|
=|\bm e_m(\hat x)\cdot E_2^\infty(\hat x,d_1,d_2,p_1,p_2)|
\en
for all $\hat x\in\Sp^2\se\{N,S\}$, $d_1,d_2\in\Sp^2$, $m\in\{\phi,\theta\}$ and $p_1,p_2\in\R^3$
satisfying that $d_1\bot p_1$ and $d_2\bot p_2$, then $D_1=D_2$ and $\mathscr B_1=\mathscr B_2$.
\end{theorem}

To prove Theorem \ref{o}, we need the following lemma.

\begin{lemma}\label{conj*}
Under the assumptions of Theorem $\ref{o}$, the following equation does not hold:
\be\label{conj}
E_1^\infty(\hat x,d_0)\bm e_m(d_0)=e^{i\beta}\ov{E_2^\infty(\hat x,d_0)}\bm e_m(d_0)
\;\;\;\forall\hat x\in\Sp^2,
\en
where $m\in\{\phi,\theta\}$ and $d_0\in\Sp^2\se\{N,S\}$ are arbitrarily fixed and
$\beta$ is a real constant.
\end{lemma}

\begin{proof}
Assume to the contrary that (\ref{conj}) holds. Then, by using the Stratton--Chu formula
(see \cite[Theorem 6.7]{CK}), we have that the electric scattered field $E_2^s(x,d_0)\bm e_m(d_0)$
satisfies that
\ben
E_2^s(x,d_0)\bm e_m(d_0)&=&{\rm curl}\,\int_{\pa B\cup\pa B_R}\nu(y)\times
    [E_2^s(y,d_0)\bm e_m(d_0)]\Phi(x,y)ds(y)\\
&&-\frac1{ik}{\rm curl}\,{\rm curl}\,\int_{\pa B\cup\pa B_R}\nu(y)\times [H_2^s(y,d_0)\bm e_m(d_0)]\Phi(x,y)ds(y),\;x\in\R^3\se\ov{B\cup B_R},
\enn
where $\nu(y)$ is the unit normal vector at $y\in\pa B$ or $y\in\pa B_R$ directed into the exterior of
$B$ or $B_R$ and $\Phi(x,y)$ is the fundamental solution to the Helmholtz equation in $\R^3$ given by
\ben
\Phi(x,y):=\frac1{4\pi}\frac{e^{ik|x-y|}}{|x-y|},\;\;\;x\neq y.
\enn
Then it follows from \cite[(6.25)]{CK} that the corresponding far-field pattern
$E_2^\infty(\hat x,d_0)\bm e_m(d_0)$ is given as
\ben
&&E_2^\infty(\hat x,d_0)\bm e_m(d_0)\\
&&\quad=\frac{ik}{4\pi}\hat x\times\int_{\pa B\cup\pa B_R}
\left\{\nu(y)\times [E_2^s(y,d_0)\bm e_m(d_0)]+\left[\nu(y)\times [H_2^s(y,d_0)\bm e_m(d_0)]\right]
\times\hat x\right\}e^{-ik\hat x\cdot y}ds(y),\;\;\hat x\in\Sp^2.\qquad
\enn
By this and (\ref{conj}) we deduce that for any $\hat x\in\Sp^2$,
\ben
&&e^{-i\beta}E_1^\infty(\hat x,d_0)\bm e_m(d_0)=\ov{E_2^\infty(\hat x,d_0)}\bm e_m(d_0)\\
&&\;=-\frac{ik}{4\pi}\hat x\times\int_{\pa B\cup\pa B_R}\left\{\nu(y)\times[\ov{E_2^s(y,d_0)}
\bm e_m(d_0)]+\left[\nu(y)\times[\ov{H_2^s(y,d_0)}\bm e_m(d_0)]\right]
\times\hat x\right\}e^{ik\hat x\cdot y}ds(y)\\
&&\;=-\frac{ik}{4\pi}\hat x\times\int_{\pa\wid{B}\cup\pa B_R}\left\{-\nu(y)
\times[\ov{E_2^s(-y,d_0)}\bm e_m(d_0)]+\left[-\nu(y)
\times[\ov{H_2^s(-y,d_0)}\bm e_m(d_0)]\right]\times\hat x\right\}e^{-ik\hat x\cdot y}ds(y)\\
&&\;=\frac{ik}{4\pi}\hat x\times\int_{\pa\wid{B}\cup\pa B_R}\left\{\nu(y)
\times[\ov{E_2^s(-y,d_0)}\bm e_m(d_0)]+\left[\nu(y)
\times[\ov{H_2^s(-y,d_0)}\bm e_m(d_0)]\right]\times\hat x\right\}e^{-ik\hat x\cdot y}ds(y),\quad
\enn
where $\wid{B}:=\{x\in\R^3:-x\in B\}$. Then, by Rellich's lemma we have
\be\label{E1s}
e^{-i\beta}E_1^s(x,d_0)\bm e_m(d_0)&=&{\rm curl}\,\int_{\pa\wid{B}\cup\pa B_R}\nu(y)
\times[\ov{E_2^s(-y,d_0)}\bm e_m(d_0)]\Phi(x,y)ds(y)\\ \no
&&-\frac1{ik}{\rm curl}\,{\rm curl}\,\int_{\pa\wid{B}\cup\pa B_R}\nu(y)
\times[\ov{H_2^s(-y,d_0)}\bm e_m(d_0)]\Phi(x,y)ds(y),\;x\in\R^3\se\ov{\wid{B}\cup B_R}.
\en
This implies that $E_1^s(\cdot,d_0)\bm e_m(d_0)$ can be analytically extended into
$\R^3\se\ov{\widetilde B\cup B_R}$. Since
\ben
H_1^s(\cdot,d_0)\bm e_m(d_0)=\frac1{ik}{\rm curl}\,[E_1^s(\cdot,d_0)\bm e_m(d_0)],
\enn
then it follows from (\ref{E1s}) that $E_1^s(\cdot,d_0)\bm e_m(d_0)$ and
$H_1^s(\cdot,d_0)\bm e_m(d_0)$ satisfy the Maxwell equations (\ref{me1})--(\ref{me2})
in $\R^3\se\ov{\wid{B}\cup B_R}$. On the other hand, by the definitions of $E^s_1$
and $H^s_1$ it is known that $E_1^s(\cdot,d_0)\bm e_m(d_0)$ and $H_1^s(\cdot,d_0)\bm e_m(d_0)$
also satisfy the Maxwell equations (\ref{me1})--(\ref{me2}) in $\R^3\se\ov{B\cup B_R}$. Since $\ov{B}\cap\ov{B_R}=\emptyset$, then the origin $0\notin B$ and $B\cap\wid{B}=\emptyset$.
Thus it is concluded that $E_1^s(\cdot,d_0)\bm e_m(d_0)$ and $H_1^s(\cdot,d_0)\bm e_m(d_0)$ satisfy
the Maxwell equations (\ref{me1})--(\ref{me2}) in $\R^3\se\ov{B_R}$. Since the electric total
field $E_1:=E_1(\cdot,d_0)\bm e_m(d_0)=E_1^i(\cdot,d_0)\bm e_m(d_0)+E_1^s(\cdot,d_0)\bm e_m(d_0)$
and the magnetic total field $H_1:=(1/{ik})\curl\,E_1$ satisfy the perfectly conducting boundary
condition on $\pa B$, then $(E_1,H_1)$ satisfies the problem (\ref{eq6})--(\ref{eq7}).
By the fact that $k$ is not a Maxwell eigenvalue in $B$ we have that $E_1\equiv0$ in $B$,
which, together with the analyticity of the electric total field $E_1$ in $\R^3\se\ov{B_R}$, implies
that $E_1\equiv0$ in $\R^3\se\ov{B_R}$. This is a contradiction, and so (\ref{conj}) dose not hold.
The proof is complete.
\end{proof}

We are now ready to prove Theorem \ref{o}.

\begin{proof}[Proof of Theorem \ref{o}]
Using (\ref{lsp}) and (\ref{m=}), we have
\be\label{m=+}
&|\bm e_m(\hat x)\cdot[E_1^\infty(\hat x,d_1)p_1+E_1^\infty(\hat x,d_2)p_2]|
=|\bm e_m(\hat x)\cdot[E_2^\infty(\hat x,d_1)p_1+E_2^\infty(\hat x,d_2)p_2|
\en
for all $\hat x\in\Sp^2\se\{N,S\}$, $d_1,d_2\in\Sp^2$, $m\in\{\phi,\theta\}$ and $p_1,p_2\in\R^3$
satisfying that $d_1\bot p_1$ and $d_2\bot p_2$. By (\ref{inc}) we know that $E^i(x,d)d=0$ for
all $\hat x,d\in\Sp^2$, and so, from the well-posedness of the scattering problem it follows that
\be\label{dd}
E_j^\infty(\hat x,d)d=0\;\;\;\forall\hat x,d\in\Sp^2.
\en
Thus (\ref{m=+}) is equivalent to the condition
\be\label{m=*}
|\bm e_m(\hat x)\cdot[E_1^\infty(\hat x,d_1)p_1+E_1^\infty(\hat x,d_2)p_2]|
=|\bm e_m(\hat x)\cdot[E_2^\infty(\hat x,d_1)p_1+E_2^\infty(\hat x,d_2)p_2|
\en
for all $\hat x\in\Sp^2\se\{N,S\}$, $d_1,d_2\in\Sp^2$, $m\in\{\phi,\theta\}$ and $p_1,p_2\in\R^3$.
This implies that
\be\no
&&\Rt\left\{[\bm e_m(\hat x)\cdot E_1^\infty(\hat x,d_1)p_1]\times
\ov{[\bm e_m(\hat x)\cdot E_1^\infty(\hat x,d_2)p_2]}\right\}\\ \label{real=}
&&\qquad\qquad=\Rt\left\{[\bm e_m(\hat x)\cdot E_2^\infty(\hat x,d_1)p_2]
\times\ov{[\bm e_m(\hat x)\cdot E_2^\infty(\hat x,d_2)p_2]}\right\}
\en
for all $\hat x\in\Sp^2\se\{N,S\}$, $d_1,d_2\in\Sp^2$, $m\in\{\phi,\theta\}$ and $p_1,p_2\in\R^3$.

For $d,q\in\Sp^2$ and $p\in\R^3$ define $r_j(\hat x,d,q,p):=|q\cdot E_j^\infty(\hat x,d)p|,\;j=1,2$.
Then, by setting $d_1=d_2=:d$ and $p_1=p_2=:p$ in (\ref{m=*}) we have
\be\no
r_1(\hat x,d,\bm e_m(\hat x),p)&=&r_2(\hat x,d,\bm e_m(\hat x),p)=:r(\hat x,d,\bm e_m(\hat x),p)\\ \label{eq8}
&&\qquad\;\forall\hat x\in\Sp^2\se\{N,S\},\;d\in\Sp^2,\;m\in\{\phi,\theta\},\;p\in\R^3.
\en
Thus we know that
\ben
\bm e_m(\hat x)\cdot E_j^\infty(\hat x,d)p
=r(\hat x,d,\bm e_m(\hat x),p)e^{i\vartheta_j^{(m)}(\hat x,d,p)}\;\;\;
\forall\hat x\in\Sp^2\se\{N,S\},\;d\in\Sp^2,\;m\in\{\phi,\theta\},\;p\in\R^3,j=1,2,
\enn
where $\vartheta_j^{(m)}$ is a real-valued function, $j=1,2$.

Let $m\in\{\phi,\theta\}$ be arbitrarily fixed. We then prove that
\be\label{eq3}
E_1^\infty(\hat x,d)\bm e_m(d)=E_2^\infty(\hat x,d)\bm e_m(d)\;\;\;
\forall\hat x\in\Sp^2,\;d\in\Sp^2\se\{N,S\}.
\en
To do this, we distinguish between the following two cases.

{\bf Case 1.} $r(\hat x,d,\bm e_m(\hat x),p)\not\equiv0$ for $\hat x\in\Sp^2\se\{N,S\}$, $d\in\Sp^2$
and $p\in\R^3$.

In this case, by the analyticity of $\bm e_m(\hat x)\cdot E_j^\infty(\hat x,d)p$ with respect to $\hat x$,
$d$ and $p$, respectively, $j=1,2$, it is easily seen that there exist open sets $U_1\subset\Sp^2\se\{N,S\}$, $U_2\subset\Sp^2$ and $V\subset\R^3$ small enough such that
(i) $r(\hat x,d,\bm e_m(\hat x),p)\not=0$ for all $\hat x\in U_1$, $d\in U_2$ and $p\in V$, and
(ii) $\vartheta_j^{(m)}(\hat x,d,p)$ is analytic with respect to $\hat x\in U_1$, $d\in U_2$ and $p\in V$,
respectively, $j=1,2$. Then, and by (\ref{real=}) we obtain that
\be\label{cos=}
\cos[\vartheta_1^{(m)}(\hat x,d_1,p_1)-\vartheta_1^{(m)}(\hat x,d_2,p_2)]
=\cos[\vartheta_2^{(m)}(\hat x,d_1,p_1)-\vartheta_2^{(m)}(\hat x,d_2,p_2)]
\en
for all $\hat x\in U_1$, $d_1,d_2\in U_2$ and $p_1,p_2\in V$. From (\ref{cos=}) and the fact
that $\vartheta_j^{(m)}(\hat x,d,p)$ is a real-valued analytic function of $\hat x\in U_1$,
$d\in U_2$ and $p\in V$, respectively, $j=1,2$, it is derived that there holds either
\be\label{vartheta1}
\vartheta_1^{(m)}(\hat x,d_1,p_1)-\vartheta_1^{(m)}(\hat x,d_2,p_2)
=\vartheta_2^{(m)}(\hat x,d_1,p_1)-\vartheta_2^{(m)}(\hat x,d_2,p_2)+2l\pi
\en
or
\be\label{vartheta2}
\vartheta_1^{(m)}(\hat x,d_1,p_1)-\vartheta_1^{(m)}(\hat x,d_2,p_2)=-[\vartheta_2^{(m)}(\hat x,d_1,p_1)-\vartheta_2^{(m)}(\hat x,d_2,p_2)]+2l\pi
\en
for some $l\in\Z$ and for all $\hat x\in U_1$, $d_1,d_2\in U_2$ and $p_1,p_2\in V$.

For the case when (\ref{vartheta1}) holds, we have
\be\no
\vartheta_1^{(m)}(\hat x,d_1,p_1)-\vartheta_2^{(m)}(\hat x,d_1,p_1)
&=&\vartheta_1^{(m)}(\hat x,d_2,p_2)-\vartheta_2^{(m)}(\hat x,d_2,p_2)+2l\pi\\ \label{eq1}
&&\;\;\;\;\quad\forall\hat x\in U_1,d_1,d_2\in U_2,p_1,p_2\in V.
\en
Fix $d_2\in U_2$, $p_2\in V$ and define
\be\label{alpha}
\alpha^{(m)}(\hat x):=\vartheta_1^{(m)}(\hat x,d_2,p_2)-\vartheta_2^{(m)}(\hat x,d_2,p_2)
\;\;\;\forall\hat x\in U_1.
\en
Then, by (\ref{eq1}) we get
\ben
\bm e_m(\hat x)\cdot E_1^\infty(\hat x,d)p&=&r(\hat x,d,\bm e_m(\hat x),p)e^{i\vartheta_1^{(m)}(\hat x,d,p)}\\
&=&r(\hat x,d,\bm e_m(\hat x),p)e^{i\alpha^{(m)}(\hat x)+i\vartheta_2^{(m)}(\hat x,d,p)}\\
&=&e^{i\alpha^{(m)}(\hat x)}\bm e_m(\hat x)\cdot E_2^\infty(\hat x,d)p
\enn
for all $\hat x\in U_1,d\in U_2$ and $p\in V$. By the analyticity of
$\bm e_m(\hat x)\cdot E_1^\infty(\hat x,d)p
-e^{i\alpha^{(m)}(\hat x)}\bm e_m(\hat x)\cdot E_2^\infty(\hat x,d)p$
in $d\in\Sp^2$ and $p\in\R^3$, respectively, it is deduced that
\be\label{alphax}
\bm e_m(\hat x)\cdot E_1^\infty(\hat x,d)p=e^{i\alpha^{(m)}(\hat x)}\bm e_m(\hat x)
\cdot E_2^\infty(\hat x,d)p\;\;\;\;\forall\hat x\in U_1,\;d\in\Sp^2,\;p\in\R^3.
\en
Changing the variables $\hat x\ra-d$ and $d\ra-\hat x$ in (\ref{alphax}) gives
\ben
\bm e_m(-d)\cdot E_1^\infty(-d,-\hat x)p=e^{i\alpha^{(m)}(-d)}\bm e_m(-d)\cdot E_2^\infty(-d,-\hat x)p
\;\;\;\forall-d\in U_1,\;\hat x\in\Sp^2,\;p\in\R^3.
\enn
The reciprocity relation $E_j^\infty(\hat x,d)=[E_j^\infty(-d,-\hat x)]^\top$
for all $\hat x,\;d\in\Sp^2$ ($j=1,2$) (see \cite[Theorem 6.30]{CK}) leads to the result
\ben
p\cdot E_1^\infty(\hat x,d)\bm e_m(-d)=e^{i\alpha^{(m)}(-d)}p\cdot E_2^\infty(\hat x,d)\bm e_m(-d)
\;\;\;\forall-d\in U_1,\;\hat x\in\Sp^2,\;p\in\R^3.
\enn
Since $\bm e_\phi(d)=-\bm e_\phi(-d)$ and $\bm e_\theta(d)=\bm e_\theta(-d)$, we have
\be\label{m1+}
E_1^\infty(\hat x,d)\bm e_m(d)=e^{i\alpha^{(m)}(-d)}E_2^\infty(\hat x,d)\bm e_m(d)\;\;\;
\forall-d\in U_1,\;\hat x\in\Sp^2.
\en
Now, by Rellich's lemma we obtain that
\be\label{eq2}
E_1^s(x,d)\bm e_m(d)=e^{i\alpha^{(m)}(-d)}E_2^s(x,d)\bm e_m(d)\;\;\;\forall x\in G,\;-d\in U_1,
\en
where $G$ denotes the unbounded component of the complement of $B\cup D_1\cup D_2$. The perfectly
conducting boundary condition on $\pa B$ gives that
$\nu\times [E_j^s(\cdot,d)\bm e_m(d)]=-\nu\times [E^i(\cdot,d)\bm e_m(d)]$ on $\pa B$ ($j=1,2$),
which, together with (\ref{eq2}), implies that
\be\label{boundaryvalue}
-\nu\times [E^i(\cdot,d)\bm e_m(d)]=-e^{i\alpha^{(m)}(-d)}\nu\times [E^i(\cdot,d)\bm e_m(d)]
\;\;\;\text{on}\;\;\;\pa B
\en
for all $-d\in U_1$. For arbitrarily fixed $-d\in U_1$, define
$\wid{E}:=(1-e^{i\alpha^{(m)}(-d)})E^i(\cdot,d)\bm e_m(d)$ and $\wid{H}:=(1/{ik})\curl\,\wid{E}$.
Then, by (\ref{boundaryvalue}) it follows that $(\wid{E},\wid{H})$ satisfies the electromagnetic
interior boundary value problem
\ben
\begin{cases}
\ds{\rm curl}\,\wid{E}-ik\wid{H}=0 & \text{in}\;\;\;B,\\
\ds{\rm curl}\,\wid{H}+ik\wid{E}=0 & \text{in}\;\;\;B,\\
\ds\nu\times\wid{E}=0 & \text{on}\;\;\;\pa B.
\end{cases}
\enn
Since $k$ is not a Maxwell eigenvalue in $B$ and $E^i(\cdot,d)\bm e_m(d)\not\equiv0$ in $B$, we have $e^{i\alpha^{(m)}(-d)}=1$ for all $-d\in U_1$. Thus it follows from (\ref{m1+}) that
\be\label{m1*}
E_1^\infty(\hat x,d)\bm e_m(d)=E_2^\infty(\hat x,d)\bm e_m(d)\;\;\;\forall-d\in U_1,\;\hat x\in\Sp^2.
\en
By the analyticity of $E_j^\infty(\hat x,d)\bm e_m(d)$ in $d\in\Sp^2\se\{N,S\}$, $j=1,2$,
the required equation (\ref{eq3}) follows.

For the case when (\ref{vartheta2}) holds, a similar argument as above gives the result
\be\label{m2}
E_1^\infty(\hat x,d)\bm e_m(d)=e^{i\beta^{(m)}(-d)}\ov{E_2^\infty(\hat x,d)}\bm e_m(d)\;\;\;
\forall\hat x\in\Sp^2,-d\in U_1,
\en
where $\beta^{(m)}$ is a real-valued function defined by
\be\label{beta}
\beta^{(m)}(\hat x):=\vartheta_1^{(m)}(\hat x,d_2,p_2)+\vartheta_2^{(m)}(\hat x,d_2,p_2)
\en
for all $\hat x\in U_1$ and for some fixed $d_2\in U_2$, $p_2\in V$. However, by Lemma \ref{conj*}
(\ref{m2}) does not hold.

{\bf Case 2.} $r(\hat x,d,\bm e_m(\hat x),p)\equiv0$ for $\hat x\in\Sp^2\se\{N,S\}$, $d\in\Sp^2$
and $p\in\R^3$. In this case, it is easily seen that (\ref{eq3}) holds.

Finally, by (\ref{dd}), (\ref{eq3}) and the linear combination of $\bm e_\phi(d)$, $\bm e_\theta(d),d$,
and noting the arbitrariness of $m\in\{\phi,\theta\}$ in (\ref{eq3}) and
the analyticity of $E_j^\infty(\hat x,d)$ in $d\in\Sp^2$, $j=1,2$, we deduce that
\be\label{1}
E_1^\infty(\hat x,d)=E_2^\infty(\hat x,d)\;\;\;\forall\hat x,d\in\Sp^2.
\en
This, together with \cite[Theorem 7.1]{CK}, \cite[Theorem 1]{K} and \cite[Theorem 3.1]{CCM2},
implies that $D_1=D_2$ and $\mathscr B_1=\mathscr B_2$. The proof is thus complete.
\end{proof}

\section{Uniqueness for inverse medium scattering}\label{s-m}
\setcounter{equation}{0}

Assume that $B$ is the given reference ball and that $n_0\in C^{2,\g}(\R^3)$, $0<\g<1$,
is the refractive index of a given inhomogeneous medium with the support of $n_0-1$ in $\ov{B}$.
Assume further that $n_1,n_2\in C^{2,\g}(\R^3)$ are the refractive indices of two inhomogeneous media
with $m_j:=n_j-1$ supported in $\ov{D_j}$, $j=1,2$. Denote by $E_j^s$, $H_j^s$, $E_j^\infty$
and $H_j^\infty$ the electric scattered field, the magnetic scattered field, the electric far-field
pattern and the magnetic far-field pattern, respectively, associated with
the inhomogeneous medium with the refractive index $\wid{n}_j$ and corresponding to the incident
electromagnetic waves $E^i$ and $H^i$, $j=1,2$. Here, the refractive index $\wid{n}_j$ is given by
\ben
\wid{n}_j(x):=\begin{cases}
n_0(x), & x\in B,\\
n_j(x), & x\in\R^3\se\ov{B}
\end{cases}
\enn
for $j=1,2$. It is noticed that if $\ov{D}_j\cap\ov{B}=\emptyset$ then $\wid{n}_j\in C^{2,\g}(\R^3)$.
See Figure \ref{figm} for the geometric description of the scattering problem.
Suppose $k$ is not an electromagnetic interior transmission eigenvalue in $B$.
Here, $k$ is called an electromagnetic interior transmission eigenvalue in $B$ if the homogeneous
electromagnetic interior transmission problem
\be\label{eq4}
\begin{cases}
\ds{\rm curl}\,\wid E-ik\wid H=0,\;\;\;{\rm curl}\,\wid H+ikn_0\wid E=0 & \text{in}\;\;B,\\
\ds{\rm curl}\,E_0-ikH_0=0,\;\;\;{\rm curl}\,H_0+ikE_0=0 & \text{in}\;\;B,\\  
\ds\nu\times(\wid E-E_0)=0,\;\;\;\nu\times(\wid H-H_0)=0 & \text{on}\;\;\pa B
\end{cases}
\en
has a nontrivial solution $(\wid E,\wid H,E_0,H_0)$. Note that if $n_0$ is chosen so that
$\I\,[n_0(x_0)]>0$ for some $x_0\in B$ then $k$ is not an electromagnetic interior
transmission eigenvalue (see the discussion in the proof of \cite[Theorem 9.8]{CK}).

\begin{figure}[!ht]
\centering
\includegraphics[width=0.5\textwidth]{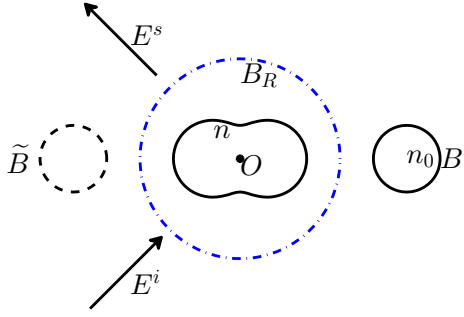}
\vspace{-1cm}
\caption{Scattering by an inhomogeneous medium.
}\label{figm}
\end{figure}

\begin{theorem}\label{m}
Assume that $B$ is a given ball filled with the inhomogeneous medium of the refractive index
$n_0\in C^{2,\g}(\R^3)$, $0<\g<1$, such that the support of $n_0-1$ is $\ov{B}$ and $k$ is not
an electromagnetic interior transmission eigenvalue in $B$.
Assume further that $n_1,n_2\in C^{2,\g}(\R^3)$ are the refractive
indices of two inhomogeneous media with $m_j:=n_j-1$ supported in $\ov{D_j}$, $j=1,2$.
Suppose $\ov{D_1\cup D_2}\subset B_R$, where $B_R$ is a ball of radius $R$ and centered at the origin
and satisfies that $\ov{B}\cap\ov{B_R}=\emptyset$. If the corresponding electric far-field patterns
satisfy (\ref{m=}) for all $\hat x\in\Sp^2\se\{N,S\}$, $d_1,d_2\in\Sp^2$, $m\in\{\phi,\theta\}$ and
$p_1,p_2\in\R^3$ satisfying that $d_1\bot p_1$ and $d_2\bot p_2$, then $n_1=n_2$.
\end{theorem}

To prove Theorem \ref{m}, we need the following lemma which is similar to Lemma \ref{conj*}.

\begin{lemma}\label{conj*m}
Under the assumptions of Theorem \ref{m}, there does not hold the equation
\be\label{conjm}
E_1^\infty(\hat x,d_0)\bm e_m(d_0)=e^{i\beta}\ov{E_2^\infty(\hat x,d_0)}\bm e_m(d_0)
\;\;\;\forall\hat x\in\Sp^2,
\en
where $m\in\{\phi,\theta\}$ and $d_0\in\Sp^2$ are arbitrarily fixed and $\beta$ is a real constant.
\end{lemma}

\begin{proof}
Assume to the contrary that (\ref{conjm}) holds.
Then, by a similar argument as in the proof of Lemma \ref{conj*} it can be derived that
$E_1^s(\cdot,d_0)\bm e_m(d_0)$ and $H_1^s(\cdot,d_0)\bm e_m(d_0)$ can be analytically extended
into $\R^3\se\ov{B_R}$ and satisfy the Maxwell equations (\ref{me1})--(\ref{me2}) in $\R^3\se\ov{B_R}$.
Noting that the incident waves $E^i(\cdot,d_0)\bm e_m(d_0)$ and $H^i(\cdot,d_0)\bm e_m(d_0)$ satisfy
the Maxwell equations (\ref{me1})--(\ref{me2}) in $\R^3$, we obtain that the total fields
$E_1:=E_1(\cdot,d_0)\bm e_m(d_0)=E_1^i(\cdot,d_0)\bm e_m(d_0)+E_1^s(\cdot,d_0)\bm e_m(d_0)$ and
$H_1:=(1/{ik}){\rm curl}\,E_1$ satisfy the Maxwell equations (\ref{me1})--(\ref{me2}) in $B$.

On the other hand, from the definition of $\wid{n}_1$ and the electromagnetic medium scattering problem
it follows that the total fields $E_1$ and $H_1$ also satisfy the first two Maxwell equations in (\ref{eq4})
in $B$. Thus, $(\wid E,\wid H,E_0,H_0):=(E_1,H_1,E_1,H_1)$ satisfies the problem (\ref{eq4}).
Since $k$ is not an electromagnetic interior transmission eigenvalue in $B$,
it follows that $E_1\equiv0$ in $B$. By the analyticity of the electric scattered field
$E_1^s(\cdot,d_0)\bm e_m(d_0)$ in $\R^3\se\ov{B_R}$, it is easily seen that $E_1$ is also analytic
in $\R^3\se\ov{B_R}$, and thus we have that $E_1\equiv0$ in $\R^3\se\ov{B_R}$.
This is a contradiction, implying that (\ref{conjm}) dose not hold. The proof is then complete.
\end{proof}

With the aid of Lemma \ref{conj*m}, we can now prove Theorem \ref{m}.

\begin{proof}[Proof of Theorem \ref{m}.]
Our proof follows a similar argument as for the case of inverse obstacle scattering in Section \ref{s-o}.
Arguing similarly as in the proof of Theorem \ref{o}, we can easily obtain (\ref{eq8}).
We now want to prove (\ref{eq3}) for arbitrarily fixed $m\in\{\phi,\theta\}$.
First, if $r(\hat x,d,\bm e_m(\hat x),p)\equiv0$ for $\hat x\in\Sp^2\se\{N,S\}$, $d\in\Sp^2$
and $p\in\R^3$, then it is obvious that (\ref{eq3}) holds.

We now consider the case $r(\hat x,d,\bm e_m(\hat x),p)\not\equiv0$ for $\hat x\in\Sp^2\se\{N,S\}$,
$d\in\Sp^2$ and $p\in\R^3$.
By a similar argument as in the proof of Theorem \ref{o} it can be shown that either (\ref{m1+})
or (\ref{m2}) holds for some open set $U_1\subset\Sp^2\se\{N,S\}$, where $\alpha^{(m)}$ and $\beta^{(m)}$
are defined similarly as in (\ref{alpha}) and (\ref{beta}), respectively, in the proof of Theorem \ref{o}.
But, Lemma \ref{conj*m} implies that (\ref{m2}) does not hold. Thus we only need to consider
the case when (\ref{m1+}) holds. By (\ref{m1+}) and Rellich's lemma we obtain (\ref{eq2}),
where $G$ is defined as above. For any fixed $-d\in U_1$, define
\ben
&&\wid{E}:=[E^i(\cdot,d)+E^s_1(\cdot,d)]\bm e_m(d)
-e^{i\alpha^{(m)}(-d)}[E^i(\cdot,d)+E^s_2(\cdot,d)]\bm e_m(d),\;\;\;
\wid H:=(1/{ik}){\rm curl}\,\wid E,\\
&& E_0:=\left(1-e^{i\alpha^{(m)}(-d)}\right)E^i(\cdot,d)\bm e_m(d),\;\;\; H_0:=(1/{ik}){\rm curl}\,E_0.
\enn
Then, by (\ref{eq2}) we have $\wid E=E_0$ in $G$, and so $(\wid E,\wid H,E_0,H_0)$ satisfies
the boundary conditions on $\pa B$ in the problem (\ref{eq4}). Further, by the definition of
$\wid{E}$, $\wid{H}$, $E_0$ and $H_0$ it is known that $(\wid E,\wid H,E_0,H_0)$ satisfies
the problem (\ref{eq4}). Since $k$ is not an electromagnetic interior transmission eigenvalue in $B$,
we obtain that $e^{i\alpha^{(m)}(-d)}=1$ for all $-d\in U_1$, which means that (\ref{m1*}) holds.
By this and the analyticity of $E_j^\infty(\hat x,d)\bm e_m(d)$ in $d\in\Sp^2\se\{N,S\}$, $j=1,2$,
it follows that (\ref{eq3}) is true.
Similarly, we can show that (\ref{1}) holds.
Since $\ov{D_1\cup D_2}\subset B_R$ and $\ov{B}\cap\ov{B_R}=\emptyset$ then $\wid{n}_j\in C^{2,\g}(\R^3)$,
$j=1,2$. Therefore, by (\ref{1}) and \cite[Theorem 4.9]{Hahner} we obtain that $n_1=n_2$.
The proof is then complete.
\end{proof}

\section{Conclusion}\label{s5}

In this paper, by adding a given reference ball into the electromagnetic scattering system,
we established uniqueness results for inverse electromagnetic obstacle and medium
scattering with phaseless electric far-field data generated by infinitely many sets of superpositions
of two electromagnetic plane waves with different directions and polarizations at a fixed
frequency for the first time.
These uniqueness results extend our previous results in \cite{XZZ2} for the acoustic case to the
electromagnetic case.
Our method is based on a simple technique of using Rellich's lemma and the Stratton--Chu formula for
the radiating solutions to the Maxwell equations.
In the future, we hope to show the same uniqueness results without using the reference ball,
which is more challenging.

\section*{Acknowledgements}

The authors thank Professor Xudong Chen at the National University of Singapore for helpful and
constructive discussions on the measurement technique of electromagnetic waves.
This work is partly supported by the NNSF of China grants 91630309, 11871466
and 11571355.

\end{document}